\documentclass{amsart}
\usepackage{hyperref}

\usepackage{graphicx}
\usepackage{amsmath,amssymb,mathrsfs}
\usepackage{amsthm}

\newtheorem{thm}{Theorem}[section]
\newtheorem*{thm*}{Theorem}
\newtheorem{cor}[thm]{Corollary}
\newtheorem{lem}[thm]{Lemma}
\newtheorem{prop}[thm]{Proposition}

\theoremstyle{definition}
\newtheorem{defn}[thm]{Definition}
\theoremstyle{remark}

\theoremstyle{example}

\theoremstyle{conjecture}

\numberwithin{equation}{section}

\newcommand{\calC}{\mathcal C}
\newcommand{\calE}{\mathcal E}

\newcommand{\calH}{\mathcal H}
\newcommand{\calS}{\mathcal S}
\newcommand{\calJ}{\mathcal J}

\newcommand {\C} {\mathbb C}
\newcommand {\R} {\mathbb R}

\newcommand {\N} {\mathbb N}
\newcommand {\D} {\mathbb D}

\begin{document}

\title[Differential Operators]{Complex symmetry of first-order differential operators on Hardy space}%

\date{\today}%

\author{Pham Viet Hai}%
\address[P. V. Hai]{Faculty of Mathematics, Mechanics and Informatics, Vietnam National University, 334 Nguyen Trai, Thanh Xuan, Hanoi, Vietnam.}%
\email{phamviethai86@gmail.com}




\maketitle

\begin{abstract}
Given holomorphic functions $\psi_0$ and $\psi_1$, we consider first-order differential operators acting on Hardy space, generated by the formal differential expression $E(\psi_0,\psi_1)f(z)=\psi_0(z)f(z)+\psi_1(z)f'(z)$. We characterize these operators which are complex symmetric with respect to weighted composition conjugations. In parallel, as a basis of comparison, a characterization for differential operators which are hermitian is carried out. Especially, it is shown that hermitian differential operators are contained properly in the class of $\calC$-selfadjoint differential operators. The calculation of the point spectrum of some differential operators is performed in detail.
\end{abstract}

\section{Introduction}
\subsection{Complex symmetric operators}
Due to the applicability in various fields, the study of complex symmetric operators, initiated by Garcia and Putinar in \cite{GP1, GP2}, has attracted the attention of many researchers. The general works \cite{GP1, GP2} have since stimulated intensive research on complex symmetric operators. A number of other authors have recently made significant contributions to both theory and applications (see \cite{GPP}).

We pause for a moment to recall some terminologies. 
\begin{defn}
An unbounded linear operator $T$ is called \emph{$\calC$-symmetric} on a separable complex Hilbert space $\calH$ if there exists a conjugation $\calC$ (i.e. an anti-linear, isometric involution) such that
\begin{equation}\label{C-symmetry}
\langle \calC x,Ty\rangle=\langle \calC Tx,y\rangle,\quad\forall x,y\in\text{dom}(T).
\end{equation}
\noindent Note that for a densely defined operator $T$, its adjoint $T^*$ is well-defined, and so the identity \eqref{C-symmetry} means that $T\preceq \calC T^*\calC$. A densely defined, linear operator $S$ is called \emph{$\calC$-selfadjoint} if it satisfies $S=\calC S^*\calC$.
\end{defn}

Recently, complex symmetry has been investigated on Hilbert spaces of holomorphic function. The initial works \cite{GH,JKKL} were dedicated to bounded (weighted) composition operators acting on Hardy spaces, for a specific conjugation $\mathcal{J}f(z)=\overline{f(\overline{z})}$. The operator $\calJ$ inspired the authors \cite{HK1, HK2} to characterize its generalization, namely anti-linear weighted composition operator, which is a conjugation. With these conjugations (often called \emph{weighted composition conjugations}), the authors \cite{HK1, HK2} succeeded in characterizing bounded weighted composition operators which are complex symmetric on Fock spaces. 

\subsection{Differential operators}
Differential operators are important in a number of fields not only for their obvious applications but also as examples of linear operators. For instance, in operator theory, complex symmetric differential operators can carry a weak form of spectral decomposition theorem. This result was explored previously in the case of matrices by Takagi. (see \cite{GPP}). 

What make differential operators important in the theory of dynamical systems is the fact that they can generate $C_0$-semigroups on function spaces. In this field, we refer the reader to the paper \cite{HK4} for $C_0$-semigroups on Fock space, and to the papers \cite{avicou2015class, avicou2016analyticity} for $C_0$-semigroups on Hardy space.

Among differential operators, we can mention $\mathcal{PT}$-symmetric operators, first proposed in quantum mechanics by physicist Bender and former graduate student Boettcher \cite{bender1998real}. Roughly speaking, $\mathcal{PT}$-symmetric operators are those operators on Lebesgue space $L^2(\R)$ complex symmetric with respect to the conjugation
$$
\mathcal{PT}f(x)=\overline{f(-x)},\quad f\in L^2(\R).
$$
After the seminal work \cite{bender1998real}, there has been a rapidly growing interest in $\mathcal{PT}$-symmetric operators. Through a series of works, the study of $\mathcal{PT}$-symmetric operators has reached a certain state of advancement, which is well represented in the survey \cite{konotop2016nonlinear}.

Recently, the authors \cite{Hai2018} have discovered that the study of $\mathcal{PT}$-symmetric operators can be linked to the complex symmetry in Fock space. It turns out that $\mathcal{PT}$ is unitarily equivalent to $\mathcal{J}$ acting on Fock space (see \cite[Proposition 5.5]{Hai2018}). This result allows us to show that a maximal differential operator is $\mathcal{PT}$-selfadjoint on the Lebesgue space if and only if it is unitarily equivalent to a linear combination  with complex coefficients of
operators (acting on Fock space) of the form
$$ z^m [\frac{\partial}{\partial z}]^n + (-1)^{m+n}  z^n [\frac{\partial}{\partial z}]^m.$$
The linear combination is endowed with the maximal domain of definition.

In this paper, we concentrate on first-order differential operators acting on Hardy space. For this purpose, it is essential to recall the basic definitions. The Hardy space $H^2$ consists of the holomorphic functions in $\D$, whose mean square value on the circle of radius $r$ remains bounded as $r\to 1$ from below. On this space, we consider the formal differential expression of the form
$$
E(\psi_0,\psi_1)f(z)=\psi_0(z)f(z)+\psi_1(z)f'(z),\quad f\in H^2,
$$
where holomorphic functions $\psi_0,\psi_1:\D\to\C$ are called as the \emph{symbols}. We define the \emph{maximal differential operator} $\calE_{\max}$ as follows
$$
\text{dom}(\calE_{\max})=\{f\in H^2: E(\psi_0,\psi_1)f\in H^2\},\quad \calE_{\max}f=E(\psi_0,\psi_1)f.
$$
The operator $\calE_{\max}$ is ``maximal" in the sense that one cannot extend it as a linear operator in $H^2$ generated by $E(\psi_0,\psi_1)$. The operator $\calE$ is called a \emph{first-order differential operator} associated with the expression $E(\psi_0,\psi_1)$ if $\calE\preceq \calE_{\max}$, namely
$$
\text{dom}(\calE)\subseteq\text{dom}(\calE_{\max}),\quad \calE f=E(\psi_0,\psi_1)f.
$$

\subsection{Content}
This paper is interested in how the symmetry properties of the first-order differential operators affect the function theoretic properties of the symbols, and vice versa. In Theorems \ref{Cself-c1} and \ref{Cself-2}, we characterize \emph{maximal} differential operators which are $\calC$-selfadjoint on the Hardy space, with respect to weighted composition conjugations. A similar calculation is carried out for hermitian operators in Theorem \ref{selfadjoint}. Meanwhile, Theorems \ref{no-max-do}, \ref{no-max-do-c2} and \ref{no-max-do-self-her} are devoted to the study of differential operators with \emph{arbitrary} domains (not necessarily maximal). These results show that there is no nontrivial domain for a first-order differential operator $T$ on which $T$ is hermitian as well as $\calC$-selfadjoint with respect to some weighted composition conjugation. Especially, Corollary \ref{cor-her-cso} shows that hermitian differential operators are contained properly in the class of $\calC$-selfadjoint differential operators. The point spectrum of some differential operators is computed in detail.

\section*{Notations}
Throughout the paper, we let $\N$, $\R$, $\C$ denote the set of non-negative integers, the set of real numbers and the set of complex numbers, respectively. Let $\D=\{z\in\C:|z|<1\}$ and $\partial\D=\{z\in\C:|z|=1\}$. The domain of an unbounded operator is denoted as $\text{dom}(\cdot)$. For two unbounded operators $X,Y$, the notation $X\preceq Y$ means that $X$ is the \emph{restriction} of $Y$ on the domain $\text{dom}(X)$; namely
$$
\text{dom}(X)\subseteq\text{dom}(Y),\quad Xz=Yz,\,\forall z\in\text{dom}(X).
$$

\section{Preliminaries}
In this section, we gather some properties of the Hardy space, which are used in the later proofs. Given $f\in H^2$, the radial limit $\lim\limits_{r\to 1^{-}}f(re^{i\theta})$ always exists for almost every $\theta$ and we denote it as $f(e^{i\theta})$. It is well-known that $H^2$ is a reproducing kernel Hilbert space with the inner product
$$\langle f,g\rangle=\dfrac{1}{2\pi}\int\limits_0^{2\pi}f(e^{i\theta})\overline{g(e^{i\theta})}\,d\theta,$$
and with kernel functions
$$
K_z(u)=\dfrac{1}{1-\overline{z}u}.
$$
We denote
$$K_z^{(m)}(u)=\dfrac{\overline{z}^m m!}{(1-\overline{z}u)^{m+1}},\quad K_z^{[m]}(u)=\dfrac{u^m m!}{(1-\overline{z}u)^{m+1}}.
$$
Note that these functions satisfy
$$
K_z^{(m)}(u)=\dfrac{d^m K_z(u)}{du^m},\quad f^{(m)}(z)=\langle f,K_z^{[m]}\rangle,\quad\forall m\in\N,\forall f\in H^2,\forall z,u\in\D.
$$
Two classes of conjugations on the Hardy space $H^2$ were found in \cite{Hai2018-stolen}. The first class consists of conjugations defined by
$$
\calC_{\alpha,\beta}f(z)=\alpha\overline{f(\beta\overline{z})},\quad \alpha,\beta\in\partial\D,
$$
and the other one contains the following conjugations
$$
\calJ_{\beta,\lambda}f(z)=\dfrac{\beta\sqrt{1-|\lambda|^2}}{1-z\overline{\lambda}}\cdot\overline{f\left(\dfrac{\lambda}{\overline{\lambda}}\cdot\dfrac{\overline{\lambda}-\overline{z}}{1-\lambda\overline{z}}\right)},\quad\beta\in\partial\D,\lambda\in\D\setminus\{0\}.
$$
With the same techniques, the authors \cite{lim2018complex} reformulated these conjugations on weighted Hardy spaces.

\section{Some initial properties}
\subsection{Reproducing kernel algebra and closed graph}
For our main results, we need the following simple but basic observations.

The first observation is related to the action of the adjoint of a differential operator on the kernel functions, that is on the point evaluation functionals.
\begin{lem}\label{bcajkf}
For every $w\in\D$, $m\in\N$, we have $K_w^{[m]}\in\text{dom}(\calE^*)$, and moreover
\begin{eqnarray}\label{T*K_z}
\calE^*K_w^{[m]}
\nonumber&=&\overline{\psi_0^{(m)}(w)}K_w+\sum_{j=1}^m\left[ \binom{m}{j}\overline{\psi_0^{(m-j)}(w)}+\binom{m}{j-1}\overline{\psi_1^{(m-j+1)}(w)} \right]K_w^{[j]}\\
&&+\overline{\psi_1(w)}K_w^{[m+1]}.
\end{eqnarray}
\end{lem}
\begin{proof}
We use a proof by induction on $m$. For any $g\in\text{dom}(\calE)$, we have
\begin{eqnarray*}
\langle \calE g,K_w\rangle &=& (\calE g)(w)=\psi_0(w)g(w)+\psi_1(w)g'(w)=\langle g,\overline{\psi_0(w)}K_w+\overline{\psi_1(w)}K_w^{[1]}\rangle,
\end{eqnarray*}
which gives \eqref{T*K_z} when $m=0$. Suppose that \eqref{T*K_z} holds for $m=\ell$, and hence
\begin{eqnarray*}
\langle g,\calE^*K_w^{[\ell]}\rangle 
&=&\psi_0^{(\ell)}(w)g(w)+\sum_{j=1}^\ell\left[ \binom{\ell}{j}\psi_0^{(\ell-j)}(w)+\binom{\ell}{j-1}\psi_1^{(\ell-j+1)}(w) \right]g^{(j)}(w)\\
&&+\psi_1(w)g^{(\ell+1)}(w).
\end{eqnarray*}
Since $\langle g,\calE^*K_w^{[\ell]}\rangle=\langle \calE g,K_w^{[\ell]}\rangle=(\calE g)^{(\ell)}(w)$, we have
\begin{eqnarray*}
(\calE g)^{(\ell)}(w)
&=&\psi_0^{(\ell)}(w)g(w)+\sum_{j=1}^\ell\left[ \binom{\ell}{j}\psi_0^{(\ell-j)}(w)+\binom{\ell}{j-1}\psi_1^{(\ell-j+1)}(w) \right]g^{(j)}(w)\\
&&+\psi_1(w)g^{(\ell+1)}(w).
\end{eqnarray*}
Differentiating both sides gives
\begin{eqnarray*}
(\calE g)^{(\ell+1)}(w)
&=&\psi_0^{(\ell+1)}(w)g(w)+\sum_{j=1}^{\ell+1}\left[ \binom{\ell+1}{j}\psi_0^{(\ell+1-j)}(w)+\binom{\ell+1}{j-1}\psi_1^{(\ell-j+2)}(w) \right]g^{(j)}(w)\\
&&+\psi_1(w)g^{(\ell+2)}(w),
\end{eqnarray*}
which implies, as $(\calE g)^{(\ell+1)}(w)=\langle \calE g,K_w^{[\ell+1]}\rangle$, that
\begin{eqnarray*}
\langle \calE g,K_w^{[\ell+1]}\rangle
&=&\langle g,\overline{\psi_0^{(\ell+1)}(w)}K_w+\sum_{j=1}^{\ell+1}\left[ \binom{\ell+1}{j}\overline{\psi_0^{(\ell+1-j)}(w)}+\binom{\ell+1}{j-1}\overline{\psi_1^{(\ell-j+2)}(w)} \right]K_w^{[j]}\\
&&+\overline{\psi_1(w)}K_w^{[\ell+2]}\rangle.
\end{eqnarray*}
The above identity shows $K_w^{[\ell+1]}\in\text{dom}(\calE^*)$ and \eqref{T*K_z} holds for $m=\ell+1$.
\end{proof}

The next observation shows that a maximal differential operator is always closed.
\begin{prop}\label{T-closed}
The operator $\calE_{\max}$ is always closed.
\end{prop}
\begin{proof}
Let $(f_n)_{n\in\N}\subset\text{dom}(\calE_{\max})$ and $f,g\in H^2$, such that
$$
f_n\to f \quad\text{and}\quad \calE_{\max}f_n\to g\quad\hbox{in $H^2$},
$$
which imply, as $f_n(z)=\langle f_n,K_z\rangle$ and $f_n'(z)=\langle f_n,K_z^{[1]}\rangle$, that
$$
f_n(z)\to f(z),\quad f_n'(z)\to f'(z) \quad\text{and}\quad \calE_{\max}f_n(z)\to g(z),\quad\forall z\in\D.
$$
Letting $n\to\infty$ in the identity $\calE_{\max}f_n(z)=\psi_0(z)f_n(z)+\psi_1(z)f_n'(z)$ gives
$$
\psi_0(z)f(z)+\psi_1(z)f'(z)=g(z),\quad\forall z\in \D.
$$
Since $g\in H^2$, we conclude that $f\in\text{dom}(\calE_{\max})$, and furthermore $\calE_{\max}f=g$.
\end{proof}

\subsection{Adjoints}
As it will be seen in the next section, for a $\calC$-selfadjoint differential operator,  all its symbols have to be polynomials. It is essential to explore the adjoint of a differential operator when the symbols are polynomials.
\begin{thm}\label{adjoint-form}
Let $\calE_{\max}$ be a maximal differential operator, induced by the symbols
$$
\psi_0(z)=a+zb,\quad\psi_1(z)=d+zc+z^2b,\quad z\in\D,
$$
where $a,b,c\in\C$. Then we always have $\calE_{\max}^*=\calS_{\max}$, where $\calS_{\max}$ is the maximal differential operator induced by the symbols
$$
\widehat{\psi_0}(z)=\overline{a}+z\overline{d},\quad \widehat{\psi_1}(z)=\overline{b}+z\overline{c}+z^2\overline{d},\quad z\in\D.
$$
\end{thm}
\begin{proof}
First, we show that $\calE_{\max}^*\preceq \calS_{\max}$.

Take arbitrarily $z,u\in\D$ and $f\in\text{dom}(\calE_{\max}^*)$. On one hand, we have
\begin{eqnarray*}
\calE_{\max}K_z(u) 
=\dfrac{a+ub}{1-u\overline{z}}+\dfrac{(d+uc+u^2b)\overline{z}}{(1-u\overline{z})^2}
=(d\overline{z}+a)K_z(u)+(d\overline{z}^2+c\overline{z}+b)K_z^{[1]}(u),
\end{eqnarray*}
and hence
\begin{eqnarray*}
\langle f,\calE_{\max}K_z\rangle=(z\overline{d}+\overline{a})f(z)+(z^2\overline{d}+z\overline{c}+\overline{b})f'(z)=\widehat{\psi_0}(z)f(z)+\widehat{\psi_1}(z)f'(z).
\end{eqnarray*}
On the other hand, by the definitions of adjoint operators and kernel functions,
$$
\langle f,\calE_{\max}K_z\rangle =\langle \calE_{\max}^*f,K_z\rangle = \calE_{\max}^*f(z).
$$
These show that $\calE_{\max}^*f(z)=\widehat{\psi_0}(z)f(z)+\widehat{\psi_1}(z)f'(z)=\calS_{\max}f(z)$. Due to the arbitrariness of $f$ in $\text{dom}(\calE_{\max}^*)$, we get $\calE_{\max}^*\preceq \calS_{\max}$.

Next, we show that $\calS_{\max}\preceq \calE_{\max}^*$. It is enough to prove that
$$
\langle \calE_{\max}f,g\rangle=\langle f,\calS_{\max}g\rangle,\quad\forall f\in\text{dom}(\calE_{\max}),\forall g\in\text{dom}(\calS_{\max}).
$$
Take arbitrarily $f\in\text{dom}(\calE_{\max})$ and $g\in\text{dom}(\calS_{\max})$. We have
\begin{eqnarray*}
\int\limits_{-\pi}^\pi \calE_{\max}f(e^{i\theta})\overline{g(e^{i\theta})}\,d\theta
&=&\int\limits_{-\pi}^\pi (a+be^{i\theta})f(e^{i\theta})\overline{g(e^{i\theta})}\,d\theta\\
&&+\int\limits_{-\pi}^\pi (d+ce^{i\theta}+be^{2i\theta})f'(e^{i\theta})\overline{g(e^{i\theta})}\,d\theta.
\end{eqnarray*}
The second integral of the right-hand side is rewritten as
\begin{eqnarray*}
&&\int\limits_{-\pi}^\pi (d+ce^{i\theta}+be^{2i\theta})f'(e^{i\theta})\overline{g(e^{i\theta})}\,d\theta\\
&&=\int\limits_{-\pi}^\pi (de^{-i\theta}+c+be^{i\theta})\overline{g(e^{i\theta})}\,\dfrac{df(e^{i\theta})}{i}\\
&&=\int\limits_{-\pi}^\pi f(e^{i\theta})\left[(de^{-i\theta}-be^{i\theta})\overline{g(e^{i\theta})}+(de^{-2i\theta}+ce^{-i\theta}+b)\overline{g'(e^{i\theta})}\right]\,d\theta.
\end{eqnarray*}
Thus, we get
\begin{eqnarray*}
\int\limits_{-\pi}^\pi \calE_{\max}f(e^{i\theta})\overline{g(e^{i\theta})}\,d\theta
=\int\limits_{-\pi}^\pi f(e^{i\theta})\left[\overline{(de^{i\theta}+a)g(e^{i\theta})+(de^{2i\theta}+ce^{i\theta}+b)g'(e^{i\theta})}\right]\,d\theta,
\end{eqnarray*}
which gives the desired conclusion.
\end{proof}

\section{Complex symmetry with respect to $\calC_{\alpha,\beta}$}
In this section, we characterize first-order differential operators which are $\calC$-selfadjoint with respect to the conjugation
\begin{equation}\label{conju-c1}
\calC_{\alpha,\beta}f(z)=\alpha\overline{f(\beta\overline{z})},\quad f\in H^2,
\end{equation}
where $\alpha,\beta\in\partial\D$. To simplify the term, these operators are called \emph{$\calC_{\alpha,\beta}$-selfadjoint}.

For the necessary condition, we only apply the symmetric condition $\calC_{\alpha,\beta}\calE\calC_{\alpha,\beta}=\calE^*$ to kernel functions $K_z$, $z\in\D$. This progress leads to the use of Lemma \ref{psi-form-c1}. When proving the sufficient condition, some care must be taken. This is due to the fact that ``a bounded operator which is complex symmetric on kernel functions, is necessarily complex symmetric on the whole Hardy space" is no longer true for differential operators. In other words, the equality that
$$
\calC_{\alpha,\beta}\calE\calC_{\alpha,\beta}K_z=\calE^*K_z,\quad\forall z\in\D
$$
cannot ensure that two unbounded operators $\calC_{\alpha,\beta}\calE\calC_{\alpha,\beta}$ and $\calE^*$ are identical. To prove that $\calC_{\alpha,\beta}\calE\calC_{\alpha,\beta}=\calE^*$, we take turns solving the following tasks: (1) compute the adjoint $\calE^*$; (2) identify the operator $\calC_{\alpha,\beta}\calE\calC_{\alpha,\beta}$. The answer to the task (1) is derived from Theorem \ref{adjoint-form}, while the one to the task (2) lies in Proposition \ref{conj-D}.

\subsection{Auxiliary lemmas}
The following observation is a simple computation, but it will be an important step toward identifying the operator $\calC_{\alpha,\beta}\calE_{\max}\calC_{\alpha,\beta}$.
\begin{prop}\label{conj-D}
The identity $\calC_{\alpha,\beta}E(\psi_0,\psi_1)\calC_{\alpha,\beta}=E(\widehat{\psi_0},\widehat{\psi_1})$ holds, where
$$
\widehat{\psi_0}(z)=\overline{\psi_0(\beta\overline{z})},\,\,\widehat{\psi_1}(z)=\overline{\psi_1(\beta\overline{z})}\beta,\quad z\in\D.
$$
\end{prop}
\begin{proof}
We have
\begin{eqnarray*}
E(\psi_0,\psi_1)\calC_{\alpha,\beta}f(z)
&=&\psi_0(z)\alpha\overline{f(\beta\overline{z})}+\psi_1(z)\alpha\overline{\beta f'(\beta\overline{z})},
\end{eqnarray*}
which implies, as $|\alpha|=|\beta|=1$, that
\begin{eqnarray*}
\calC_{\alpha,\beta}E(\psi_0,\psi_1)\calC_{\alpha,\beta}f(z)
&=&\overline{\psi_0(\beta\overline{z})}f(z)+\overline{\psi_1(\beta\overline{z})}\beta f'(z)=E(\widehat{\psi_0},\widehat{\psi_1}).
\end{eqnarray*}
\end{proof}

Next, we give a necessary condition for a maximal differential operator of order $1$ when it is $\calC_{\alpha,\beta}$-selfadjoint on the Hardy space. To do that, we need the below lemma to compute the symbols.
\begin{lem}\label{psi-form-c1}
Suppose that the holomorphic functions $\psi_0$ and $\psi_1$ in $\D$ satisfy
\begin{equation}\label{important-c1}
\psi_0(u)+\dfrac{\psi_1(u)\overline{z}}{1-u\overline{z}}=\psi_0(\beta\overline{z})+\dfrac{\psi_1(\beta\overline{z})u\overline{\beta}}{1-u\overline{z}},\quad\forall z,u\in\D.
\end{equation}
Then
\begin{equation}\label{form-psi-c1}
\psi_0(u)=a+ub,\quad\psi_1(u)=b\beta+uc+u^2b,\quad u\in\D,
\end{equation}
where $a,b,c\in\C$ are constants.
\end{lem}
\begin{proof}
Indeed, letting $z=0$ in \eqref{important-c1} gives $\psi_0(u)=\psi_0(0)+\psi_1(0)u\overline{\beta}$. Substitute $\psi_0$ back into \eqref{important-c1} to obtain
$$
\dfrac{\psi_1(u)}{1-u\overline{z}}-\psi_1(0)=\dfrac{1}{\overline{z}}\left[\dfrac{\psi_1(\beta\overline{z})}{1-u\overline{z}}-\psi_1(0)\right]u\overline{\beta}
$$
In the above identity, we let $z\to 0$ to get the desired conclusion.
\end{proof}

The necessary condition was given by the following proposition.
\begin{prop}\label{=>-c1}
Let $\calE_{\max}$ be a maximal differential operator of order $1$, induced by the holomorphic functions $\psi_0$ and $\psi_1$, and $\calC_{\alpha,\beta}$ the weighted composition conjugation given by \eqref{conju-c1}. If the identity $\calE_{\max}K_z=\calC_{\alpha,\beta}\calE_{\max}^*\calC_{\alpha,\beta}K_z$ holds for every $z\in\D$, then the symbols are of the forms \eqref{form-psi-c1}.
\end{prop}
\begin{proof}
Note that
$$
\calE_{\max}K_z(u)=\dfrac{\psi_0(u)}{1-u\overline{z}}+\dfrac{\psi_1(u)\overline{z}}{(1-u\overline{z})^2}.
$$
and
$$
\calC_{\alpha,\beta}\calE_{\max}^*\calC_{\alpha,\beta}K_z(u)=\dfrac{\psi_0(\beta\overline{z})}{1-u\overline{z}}+\dfrac{\psi_1(\beta\overline{z})u\overline{\beta}}{(1-u\overline{z})^2}.
$$
Thus, the identity $\calE_{\max}K_z=\calC_{\alpha,\beta}\calE_{\max}^*\calC_{\alpha,\beta}K_z$ is reduced to \eqref{important-c1}, and hence, by Lemma \ref{psi-form-c1}, we get \eqref{form-psi-c1}.
\end{proof}

\subsection{Maximal domain}
With all preparation discussed in the preceding subsection, we now characterize maximal differential operators of order $1$, which are $\calC_{\alpha,\beta}$-selfadjoint on the Hardy space $H^2$.
\begin{thm}\label{Cself-c1}
Let $\calC_{\alpha,\beta}$ be the weighted composition conjugation given by \eqref{conju-c1}, and $\calE_{\max}$ the maximal differential operator of order $1$, induced by the holomorphic functions $\psi_0$ and $\psi_1$. The following assertions are equivalent.
\begin{enumerate}
\item The operator $\calE_{\max}$ is $\calC_{\alpha,\beta}$-selfadjoint.
\item The domain $\text{dom}(\calE_{\max})$ is dense and $\calE_{
\max}^*\preceq\calC_{\alpha,\beta}\calE_{\max}\calC_{\alpha,\beta}$.
\item The symbols are of the forms \eqref{form-psi-c1}, that is
$$
\psi_0(u)=a+ub,\quad\psi_1(u)=b\beta+uc+u^2b,\quad u\in\D,
$$
where $a,b,c\in\C$.
\end{enumerate}
\end{thm}
\begin{proof}
It is obvious that $(1)\Longrightarrow(2)$, while Proposition \ref{=>-c1} shows that $(2)\Longrightarrow(3)$. It rests to demonstrate that $(3)\Longrightarrow(1)$. Indeed, we suppose that (3) holds. We see from Proposition \ref{T-closed}, that the operator $\calE_{\max}$ is always closed. A direct computation can show that $\psi_0 K_z+\psi_1 K_z'$ is a linear combination of elements $K_z^{[\ell]}$, and hence it must belong to the domain $\text{dom}(\calE_{\max})$. This means that the operator $\calE_{\max}$ is densely defined. By Proposition \ref{conj-D}(1) and Theorem \ref{adjoint-form}, we obtain $\calC_{\alpha,\beta}\calE_{\max}\calC_{\alpha,\beta}=\calS_{\max}=\calE_{\max}^*$, where $\calS_{\max}$ is the maximal differential operator of order $1$, induced by the symbols $\widehat{\psi_0}(z)=z\overline{b\beta}+\overline{a}$ and $\widehat{\psi_1}(z)=z^2\overline{b\beta}+z\overline{c}+\overline{b}$.
\end{proof}

\subsection{Non-maximal domain}
In the previous subsection, we investigated which the symbols $\psi_0$ and $\psi_1$ give rise the complex symmetry of maximal differential operators. Now we relax the domain assumption to only explore that $\calC_{\alpha,\beta}$-selfadjointness cannot be detached from the maximal domain.
\begin{thm}\label{no-max-do}
Let $\calE$ be a first-order differential operator induced by the holomorphic functions $\psi_0$ and $\psi_1$ (note that $\calE\preceq \calE_{\max}$). Furthermore, let $\calC_{\alpha,\beta}$ the weighted composition conjugation given by \eqref{conju-c1}. Then the operator $\calE$ is $\calC_{\alpha,\beta}$-selfadjoint if and only if the following conditions hold.
\begin{enumerate}
\item $\calE=\calE_{\max}$.
\item The symbols are of the forms \eqref{form-psi-c1}, that is
$$
\psi_0(u)=a+ub,\quad\psi_1(u)=b\beta+uc+u^2b,\quad u\in\D,
$$
where $a,b,c\in\C$.
\end{enumerate}
\end{thm}
\begin{proof}
The sufficient condition holds by Theorem \ref{Cself-c1}. To prove the necessary condition, we suppose that $\calE=\calC_{\alpha,\beta}\calE^*\calC_{\alpha,\beta}$. First, we show that $\calE_{\max}$ is $\calC_{\alpha,\beta}$-selfadjoint. Indeed, since $\calE\preceq \calE_{\max}$, we have
$$
\calE_{\max}^*\preceq \calE^*=\calC_{\alpha,\beta}\calE\calC_{\alpha,\beta}\preceq\calC_{\alpha,\beta}\calE_{\max}\calC_{\alpha,\beta},
$$
which implies, as $\calC_{\alpha,\beta}$ is involutive, that $\calC_{\alpha,\beta}\calE_{\max}^*\calC_{\alpha,\beta}\preceq \calE_{\max}$. A direct computation shows that kernel functions belong to $\text{dom}(\calC_{\alpha,\beta}\calE_{\max}^*\calC_{\alpha,\beta})$, and so
$$
\calC_{\alpha,\beta}\calE_{\max}^*\calC_{\alpha,\beta}K_z =\calE_{\max}K_z,\quad\forall z\in\D.
$$
By Proposition \ref{=>-c1}, we get the assertion (2), and hence, by Theorem \ref{Cself-c1}, the operator $\calE_{\max}$ is $\calC_{\alpha,\beta}$-selfadjoint. The assertion (1) follows from the following facts
$$
\calC_{\alpha,\beta}\calE\calC_{\alpha,\beta}\preceq \calC_{\alpha,\beta}\calE_{\max}\calC_{\alpha,\beta}=\calE_{\max}^*\preceq \calE^*=\calC_{\alpha,\beta}\calE\calC_{\alpha,\beta}.
$$
\end{proof}

\section{Complex symmetry with respect to $\calJ_{\beta,\lambda}$}
The aim of this section is to characterize differential operators which are $\calC$-selfadjoint with respect to the conjugation $\calJ_{\beta,\lambda}$ (\emph{$\calJ_{\beta,\lambda}$-selfadjoint} for short). To simplify the notations, we write
\begin{equation}\label{conj-c2}
\calJ_{\beta,\lambda}f(z)=\beta\kappa(z)\overline{f(\overline{\varphi(z)})},
\end{equation}
where
$$
\kappa(z)=\dfrac{\sqrt{1-|\lambda|^2}}{1-z\overline{\lambda}},\quad\varphi(z)=\dfrac{\overline{\lambda}}{\lambda}\cdot\dfrac{\lambda-z}{1-z\overline{\lambda}},\quad z\in\D,
$$
and $\beta\in\partial\D,\lambda\in\D\setminus\{0\}$.

\subsection{Auxiliary lemmas}
Recall that for a bounded anti-linear operator $A$, its adjoint $A^*$ is a bounded anti-linear operator satisfying
\begin{equation}\label{anti-lin-adjoint}
\langle Ax,y\rangle=\langle A^*y,x\rangle,\quad\forall x,y\in\calH.
\end{equation}
The observation that ``an anti-linear operator is a conjugation if and only if it is both selfadjoint and unitary" mentioned in  \cite{HK2} holds for any separable complex Hilbert space. Thus, we get the identities $\calJ_{\beta,\lambda}=\calJ_{\beta,\lambda}^*=\calJ_{\beta,\lambda}^{-1}$, which help establish the action of $\calJ_{\beta,\lambda}$ on the kernel functions.
\begin{lem}\label{key-lem}
For every $z\in\D$, we have
\begin{enumerate}
\item $\calJ_{\beta,\lambda}K_z=\beta\kappa(z)K_{\overline{\varphi(z)}}$.
\item $\calJ_{\beta,\lambda}K_z^{[1]}=\beta\kappa'(z)K_{\overline{\varphi(z)}}+\beta\kappa(z)\varphi'(z)K_{\overline{\varphi(z)}}^{[1]}$.
\end{enumerate}
\end{lem}
\begin{proof}
We omit the conclusion (1) and prove the conclusion (2), as the first conclusion is rather simple. For every $f\in H^2$, we have
\begin{eqnarray*}
\langle f,\calJ_{\beta,\lambda}K_z^{[1]}\rangle
&=&\langle f,\calJ_{\beta,\lambda}^*K_z^{[1]}\rangle=\langle K_z^{[1]},\calJ_{\beta,\lambda}f\rangle=\overline{[\calJ_{\beta,\lambda}f]'(z)}\\
&=&\overline{\beta\kappa'(z)}f(\overline{\varphi(z)})+\overline{\beta\kappa(z)\varphi'(z)}f'(\overline{\varphi(z)})\\
&=&\langle f,\beta\kappa'(z)K_{\overline{\varphi(z)}}+\beta\kappa(z)\varphi'(z)K_{\overline{\varphi(z)}}^{[1]}\rangle,
\end{eqnarray*}
which gives the desired result.
\end{proof}

The following observation plays an important role in proving the sufficient condition of the main result. It allows us to compute the explicit structure of the operator $\calJ_{\beta,\lambda}\calE_{\max}\calJ_{\beta,\lambda}$.
\begin{prop}\label{conj-D-c2}
The identity $\calJ_{\beta,\lambda}E(\psi_0,\psi_1)\calJ_{\beta,\lambda}=E(\psi_0^\odot,\psi_1^\odot)$ holds, where
$$
\psi_0^\odot(z)=\overline{\psi_0(\overline{\varphi(z)})}+\dfrac{\lambda(1-z\overline{\lambda})\overline{\psi_1(\overline{\varphi(z)})}}{1-|\lambda|^2},\,\,
\psi_1^\odot(z)=-\dfrac{\lambda(1-z\overline{\lambda})^2\overline{\psi_1(\overline{\varphi(z)})}}{\overline{\lambda}(1-|\lambda|^2)},\quad z\in\D.
$$
\end{prop}
\begin{proof}
We have
\begin{eqnarray*}
E(\psi_0,\psi_1)\calJ_{\beta,\lambda}f(z)
= [\psi_0(z)\beta\kappa(z)+\psi_1(z)\beta\kappa'(z)]\overline{f(\overline{\varphi(z)})}+\psi_1(z)\beta\kappa(z)\varphi'(z)\overline{f'(\overline{\varphi(z)})},
\end{eqnarray*}
which implies, as $|\beta|=1$, that
\begin{eqnarray*}
\calJ_{\beta,\lambda}E(\psi_0,\psi_1)\calJ_{\beta,\lambda}f(z)
&=&\kappa(z)[\overline{\psi_0(\overline{\varphi(z)})\kappa(\overline{\varphi(z)})+\psi_1(\overline{\varphi(z)})\kappa'(\overline{\varphi(z)})}]f(z)\\
&&+\kappa(z)\overline{\psi_1(\overline{\varphi(z)})\kappa(\overline{\varphi(z)})\varphi'(\overline{\varphi(z)})}f'(z).
\end{eqnarray*}
Note that
$$
\kappa(z)\overline{\kappa(\overline{\varphi(z)})}=1,\,\,\kappa(z)\overline{\kappa'(\overline{\varphi(z)})}=\dfrac{\lambda(1-z\overline{\lambda})}{1-|\lambda|^2},\,\,\overline{\varphi'(\overline{\varphi(z)})}=-\dfrac{\lambda(1-z\overline{\lambda})^2}{\overline{\lambda}(1-|\lambda|^2)}.
$$
\end{proof}

Like as the $\calC_{\alpha,\beta}$-symmetry, we also need a technical lemma to compute the symbols when they give rise a $\calJ_{\beta,\lambda}$-selfadjoint differential operator.
\begin{lem}\label{fsfhs}
Let $\psi_0,\psi_1$ be the holomorphic functions in $\D$ such that the identity
\begin{equation}\label{important-2}
\dfrac{\psi_0(u)}{1-u\overline{z}}+\dfrac{\psi_1(u)\overline{z}}{(1-u\overline{z})^2}=\overline{\beta\kappa(z)}\left[\psi_0(\overline{\varphi(z)})\calJ_{\beta,\lambda}K_{\overline{\varphi(z)}}(u)+\psi_1(\overline{\varphi(z)})\calJ_{\beta,\lambda}K_{\overline{\varphi(z)}}^{[1]}(u)\right]
\end{equation}
holds for every $z,u\in\D$. Then these functions are of the following forms
\begin{equation}\label{form-psi-s2}
\psi_0(u)=a+bu,\quad\psi_1(u)=d+cu+bu^2,\quad u\in\D,
\end{equation}
where $a,b,c\in\C$ and $d=-\lambda b\overline{\lambda}^{-1}-\lambda c$.
\end{lem}
\begin{proof}
In \eqref{important-2}, we let $z=0$ to get
\begin{eqnarray*}
\psi_0(u)
&=&\overline{\beta\kappa(0)}\left[\psi_0(\lambda)\calJ_{\beta,\lambda}K_{\lambda}(u)+\psi_1(\lambda)\calJ_{\beta,\lambda}K_{\lambda}^{[1]}(u)\right]\\
&=&\overline{\kappa(0)}\left[\psi_0(\lambda)\kappa(\lambda)+\psi_1(\lambda)\kappa'(\lambda)+\psi_1(\lambda)\kappa(\lambda)\varphi'(\lambda)u\right]\\
&=&\psi_0(0)+\psi_1(\lambda)\varphi'(\lambda)u,
\end{eqnarray*}
where the second equality holds by Lemma \ref{key-lem}. For choosing $z=\lambda$, the identity \eqref{important-2} becomes
\begin{eqnarray*}
&&\dfrac{\psi_0(u)}{1-u\overline{\lambda}}+\dfrac{\psi_1(u)\overline{\lambda}}{(1-u\overline{\lambda})^2}\\
&&=\overline{\beta\kappa(\lambda)}\left[\psi_0(0)\calJ_{\beta,\lambda}K_{0}(u)+\psi_1(0)\calJ_{\beta,\lambda}K_0^{[1]}(u)\right]\\
&&=\overline{\kappa(\lambda)}\left[\psi_0(0)\kappa(0)K_\lambda(u)+\psi_1(0)\kappa'(0)K_\lambda(u)+\psi_1(0)\kappa(0)\varphi'(0)K_\lambda^{[1]}(u)\right],
\end{eqnarray*}
where the second equality uses Lemma \ref{key-lem}. Consequently, bearing in mind that $\psi_0(u)=\psi_0(0)+\psi_1(\lambda)\varphi'(\lambda)u$, we obtain
\begin{eqnarray*}
\psi_1(u)
&=&\psi_1(0)+\left[\dfrac{\psi_1(\lambda)}{\lambda(1-|\lambda|^2)}-\dfrac{\psi_1(0)}{\lambda}\right]u+\psi_1(\lambda)\varphi'(\lambda)u^2.
\end{eqnarray*}
\end{proof}

Now we make use of Lemma \ref{fsfhs} to establish the necessary condition for maximal differential operators of order $1$ to be $\calJ_{\beta,\lambda}$-selfadjoint.
\begin{prop}\label{=>s2}
Let $\calE_{\max}$ be a maximal differential operator of order $1$, induced by the holomorphic functions $\psi_0$ and $\psi_1$. If the identity $\calE_{\max}K_z=\calJ_{\beta,\lambda}\calE_{\max}^*\calJ_{\beta,\lambda}K_z$ holds for every $z\in\D$, then the symbols are of the forms \eqref{form-psi-s2}.
\end{prop}
\begin{proof}
By Lemmas \ref{key-lem}(1) and the identity \eqref{T*K_z}, we have
$$
\calE_{\max}^*\calJ_{\beta,\lambda}K_z=\beta\kappa(z)\calE_{\max}^*K_{\overline{\varphi(z)}}=\beta\kappa(z)\left[\,\overline{\psi_0(\overline{\varphi(z)})}K_{\overline{\varphi(z)}}+\overline{\psi_1(\overline{\varphi(z)})}K_{\overline{\varphi(z)}}^{[1]}\right],
$$
and hence
\begin{eqnarray*}
\calJ_{\beta,\lambda}\calE_{\max}^*\calJ_{\beta,\lambda}K_z
&=&\overline{\beta\kappa(z)}\left[\psi_0(\overline{\varphi(z)})\calJ_{\beta,\lambda}K_{\overline{\varphi(z)}}+\psi_1(\overline{\varphi(z)})\calJ_{\beta,\lambda}K_{\overline{\varphi(z)}}^{[1]}\right],
\end{eqnarray*}
where we use the anti-linearity of the conjugation $\calJ_{\beta,\lambda}$. A direct computation shows that
$$
\calE_{\max}K_z(u)=\dfrac{\psi_0(u)}{1-u\overline{z}}+\dfrac{\psi_1(u)\overline{z}}{(1-u\overline{z})^2}.
$$
Thus, the assumption $\calE_{\max}K_z=\calJ_{\beta,\lambda}\calE_{\max}^*\calJ_{\beta,\lambda}K_z$ gives the identity \eqref{important-2}, and so we can use Lemma \ref{fsfhs} to get the desired result.
\end{proof}

\subsection{Maximal domain}
With all preparation discussed in the preceding subsection, we can now state and prove the main result of this section.
\begin{thm}\label{Cself-2}
Let $\calJ_{\beta,\lambda}$ be the weighted composition conjugation given by \eqref{conj-c2}, and $\calE_{\max}$ the maximal differential operator of order $1$, induced by the holomorphic functions $\psi_0$ and $\psi_1$. The following assertions are equivalent.
\begin{enumerate}
\item The operator $\calE_{\max}$ is $\calJ_{\beta,\lambda}$-selfadjoint.
\item The domain $\text{dom}(\calE_{\max})$ is dense and $\calE_{
\max}^*\preceq\calJ_{\beta,\lambda}\calE_{\max}\calJ_{\beta,\lambda}$.
\item The symbols are of the forms \eqref{form-psi-s2}, that is
$$
\psi_0(u)=a+bu,\quad\psi_1(u)=d+cu+bu^2,\quad u\in\D,
$$
where $a,b,c\in\C$ and $d=-\lambda b\overline{\lambda}^{-1}-\lambda c$.
\end{enumerate}
\end{thm}
\begin{proof}
It is obvious that $(1)\Longrightarrow(2)$, while Proposition \ref{=>s2} shows that $(2)\Longrightarrow(3)$. It rests to demonstrate that $(3)\Longrightarrow(1)$. Indeed, suppose that (3) holds. We see from Proposition \ref{T-closed}, that the operator $\calE_{\max}$ is closed. A direct computation can show that $\psi_0 K_z+\psi_1 K_z'$ is a linear combination of elements $K_z^{[\ell]}$, and hence it must belong to the domain $\text{dom}(\calE_{\max})$. This means that the operator $\calE_{\max}$ is densely defined. Proposition \ref{conj-D-c2}(2) gives $\calJ_{\beta,\lambda}\calE_{\max}\calJ_{\beta,\lambda}=\calS_{\max}$, where $\calS_{\max}$ is the maximal differential operator of order $1$, induced by
$$
\psi_0^\odot(z)=\overline{\psi_0(\overline{\varphi(z)})}+\dfrac{\lambda(1-z\overline{\lambda})\overline{\psi_1(\overline{\varphi(z)})}}{1-|\lambda|^2},\quad 
\psi_1^\odot(z)=-\dfrac{\lambda(1-z\overline{\lambda})^2\overline{\psi_1(\overline{\varphi(z)})}}{\overline{\lambda}(1-|\lambda|^2)},\quad z\in\D.
$$
A direct computation gives $\psi_1^\odot(z)=\overline{b}+z\overline{c}+z^2\overline{d}$ and $\psi_0^\odot(z)=\overline{a}+z\overline{d}$. Thus, by Theorem \ref{adjoint-form}, $\calE_{\max}^*=\calS_{\max}=\calJ_{\beta,\lambda}\calE_{\max}\calJ_{\beta,\lambda}$.
\end{proof}

\subsection{Non-maximal domain}
In the preceding subsection, we explored the structure of a maximal differential operator which is $\calJ_{\beta,\lambda}$-selfadjoint. It is natural to study a differential operator with an \emph{arbitrary domain}. It turns out that a $\calJ_{\beta,\lambda}$-selfadjoint differential operator must be necessarily maximal.
\begin{thm}\label{no-max-do-c2}
Let $\calE$ be a first-order differential operator induced by the holomorphic functions $\psi_0$ and $\psi_1$ (note that $\calE\preceq \calE_{\max}$). Then it is $\calJ_{\beta,\lambda}$-selfadjoint if and only if the following conditions hold.
\begin{enumerate}
\item $\calE=\calE_{\max}$.
\item The symbols are of the forms \eqref{form-psi-s2}, that is
$$
\psi_0(u)=a+bu,\quad\psi_1(u)=d+cu+bu^2,\quad u\in\D,
$$
where $a,b,c\in\C$ and $d=-\lambda b\overline{\lambda}^{-1}-\lambda c$.
\end{enumerate}
\end{thm}
\begin{proof}
The proof is similar to those used in Theorem \ref{no-max-do} and it is left to the reader.
\end{proof}

\section{Hermiticity}
In this section, we present a concrete description of first-order differential operators which are hermitian on the Hardy space. Recall that a densely defined, closed linear operator $T$ is called \emph{hermitian} if $T=T^*$. Some authors prefer to use the term \emph{selfadjoint} instead of \emph{hermitian}, but this use can make confusing with the term \emph{$\calC$-selfadjoint} studied in the preceding sections.
\subsection{Auxiliary lemmas}
Before proceeding with the main results of this section, we require a few observations.
\begin{lem}\label{psi-form-ss-her}
Suppose that the holomorphic functions $\psi_0$ and $\psi_1$ in $\D$ satisfy
\begin{equation}\label{important-s-her}
\psi_0(u)+\dfrac{\psi_1(u)\overline{z}}{1-\overline{z}u}=\overline{\psi_0(z)}+\dfrac{\overline{\psi_1(z)}u}{1-\overline{z}u},\quad\forall z,u\in\D.
\end{equation}
Then
\begin{equation}\label{form-psi-s-her}
\psi_0(u)=a+bu,\quad\psi_1(u)=\overline{b}+cu+bu^2,
\end{equation}
where $b\in\C$ and $a,c\in\R$.
\end{lem}
\begin{proof}
Indeed, letting $z=0$ in \eqref{important-s-her} gives $\psi_0(u)=\overline{\psi_0(0)}+\overline{\psi_1(0)}u$, and so $\psi_0(0)\in\R$. Substitute $\psi_0$ back into \eqref{important-s-her} to obtain
$$
\dfrac{\psi_1(u)}{1-u\overline{z}}-\psi_1(0)=\dfrac{1}{\overline{z}}\left[\dfrac{u\overline{\psi_1(z)}}{1-u\overline{z}}-u\overline{\psi_1(0)}\right].
$$
In the above identity, we let $z\to 0$ to get the desired conclusion.
\end{proof}

A necessary condition for maximal differential operators of order $1$ to be hermitian is provided by the following proposition.
\begin{prop}\label{=>s-her}
Let $\calE_{\max}$ be a maximal differential operator of order $1$, induced by the holomorphic functions $\psi_0$ and $\psi_1$. If the identity $\calE_{\max}K_z=\calE_{\max}^*K_z$ holds for every $z\in\D$, then the symbols are of the forms \eqref{form-psi-s-her}.
\end{prop}
\begin{proof}
A direct computation shows that
$$
\calE_{\max}K_z(u)=\dfrac{\psi_0(u)}{1-u\overline{z}}+\dfrac{\psi_1(u)\overline{z}}{(1-u\overline{z})^2}.
$$
On the other hand, by \eqref{T*K_z},
$$
\calE_{\max}^*K_z(u)=\dfrac{\overline{\psi_0(z)}}{1-u\overline{z}}+\dfrac{u\overline{\psi_1(z)}}{(1-u\overline{z})^2}.
$$
Thus, identity $\calE_{\max}K_z=\calE_{\max}^*K_z$ is reduced to \eqref{important-s-her}, and hence by Lemma \ref{psi-form-ss-her}, we get the desired conclusion.
\end{proof}

\subsection{Maximal domain}
In this section, the conclusion in Proposition \ref{=>s-her} is also a sufficient condition for a maximal differential operator to be hermitian.

\begin{thm}\label{selfadjoint}
Let $\calE_{\max}$ be a maximal differential operator of order $1$, induced by the holomorphic functions $\psi_0$ and $\psi_1$. The following assertions are equivalent.
\begin{enumerate}
\item The operator $\calE_{\max}$ is hermitian.
\item The domain $\text{dom}(\calE_{\max})$ is dense and $\calE_{\max}^*\preceq \calE_{\max}$.
\item The symbols are of the forms \eqref{form-psi-s-her}, that is
$$
\psi_0(u)=a+bu,\quad\psi_1(u)=\overline{b}+cu+bu^2,
$$
where $b\in\C$ and $a,c\in\R$.
\end{enumerate}
\end{thm}
\begin{proof}
It is obvious that $(1)\Longrightarrow(2)$, while Proposition \ref{=>s-her} shows that $(2)\Longrightarrow(3)$. It rests  to demonstrate that $(3)\Longrightarrow(1)$. Indeed, suppose that assertion (3) holds. We see from Proposition \ref{T-closed}, that the operator $\calE_{\max}$ is closed. A direct computation can show that $\psi_0 K_z+\psi_1 K_z'$ is a linear combination of elements $K_z^{[\ell]}$, and hence it must belong to the domain $\text{dom}(\calE_{\max})$. This means that the operator $\calE_{\max}$ is densely defined. By Proposition \ref{adjoint-form}, we have $\calE_{\max}^*=\calE_{\max}$, as desired.
\end{proof}

\subsection{Non-maximal domain}
The aim of this subsection is to shows that there is no nontrivial domain for a differential operator $\calE$ on which $\calE$ is hermitian.
\begin{thm}\label{no-max-do-self-her}
Let $\calE$ be a first-order differential operator induced by the holomorphic functions $\psi_0$ and $\psi_1$ (note that $\calE\preceq \calE_{\max}$). The operator $\calE$ is hermitian if and only if the following conditions hold.
\begin{enumerate}
\item $\calE=\calE_{\max}$.
\item The symbols are of the forms \eqref{form-psi-s-her}, that is
$$
\psi_0(u)=a+bu,\quad\psi_1(u)=\overline{b}+cu+bu^2,
$$
where $b\in\C$ and $a,c\in\R$.
\end{enumerate}
\end{thm}
\begin{proof}
The sufficient condition holds by Theorem \ref{selfadjoint}. To prove the necessary condition, suppose that $\calE=\calE^*$. First, we show that $\calE_{\max}$ is also hermitian. Indeed, since $\calE\preceq \calE_{\max}$, we have $\calE_{\max}^*\preceq \calE^*=\calE\preceq \calE_{\max}$, which implies that $\calE_{\max}^*\preceq \calE_{\max}$. As proved in \eqref{T*K_z}, kernel functions always belong to $\text{dom}(\calE_{\max}^*)$, and so
$$
\calE_{\max}^*K_z =\calE_{\max}K_z,\quad\forall z\in\D.
$$
By Proposition \ref{=>s-her}, we obtain conclusion (2), and so by Proposition \ref{=>s-her}, the operator $\calE_{\max}$ is hermitian. The first conclusion follows from the following facts
$$
\calE\preceq \calE_{\max}=\calE_{\max}^*\preceq \calE^*=\calE.
$$
\end{proof}

The recent results reveal that the class of complex symmetric operators is large enough to cover all hermitian operators see for instance \cite{GPP}. It is natural to discuss how big is the class of complex symmetric differential operators characterized in Theorems \ref{Cself-c1}-\ref{Cself-2}. The following result gives the answer to this question.
\begin{cor}\label{cor-her-cso}
Every hermitian first-order differential operator is $\calC_{\alpha,\beta}$-sefladjoint, where
\begin{equation*}
\alpha=1,\quad\beta=
\begin{cases}
\overline{b}b^{-1},\quad\text{if $b\ne 0$,}\\
1,\quad\text{if $b=0$.}
\end{cases}
\end{equation*}
\end{cor}
 
\section{Point spectrum}
In this section, we concentrate only on the very restrictive category of differential operators generated by the formal expression
$$
E(\psi_0,\psi_1)f(z)=\psi_0(z)f(z)+\psi_1(z)f'(z),
$$
where the top coefficient $\psi_1$ has a zero.

In the proposition below we identify all possible eigenvalues of these operators (not necessarily maximal).
\begin{prop}\label{V-point-spec}
Let $\calE$ be a first-order differential operator (not necessarily maximal), induced by the symbols $\psi_0$ and $\psi_1$. If the top coefficient $\psi_1$ has a zero at $u$, then
$$
\sigma_p(\calE)\subseteq\left\{\psi_0(u)+\ell\psi_1'(u):\ell\in\N\right\}.
$$
\end{prop}
\begin{proof}
Let $\lambda\in\sigma_p(\calE)$. Then we can find $g\in\text{dom}(\calE)\setminus\{\mathbf{0}\}$ satisfying
\begin{equation}\label{eq-wed-1}
\lambda g(z)=\calE g(z)=\psi_0(z)g(z)+\psi_1(z)g'(z),\quad\forall z\in\D.
\end{equation}
The proof is separated into two possibilities as follows.

- If $g(u)\ne 0$, then $\lambda g(u)=\psi_0(u)g(u)+\psi_1(u)g'(u)=\psi_0(u)g(u)$, which yields $\lambda=\psi_0(u)$.

- If $g(\cdot)$ attains a zero at $u$ of order $\ell\geq 1$, then differentiating \eqref{eq-wed-1} $\ell$ times and evaluating it at the point $z=u$ gives
$$
\lambda g^{(\ell)}(u)=\sum_{j=0}^\ell\binom{\ell}{j} \psi_0^{(\ell-j)}(u)g^{(j)}(u)+\sum_{j=1}^{\ell+1}\binom{\ell}{j-1} \psi_1^{(\ell+1-j)}(u)g^{(j)}(u),
$$
which implies, as $g^{(j)}(u)=0$ for any $j\in\{0,\cdots,\ell-1\}$, that
$$
\lambda g^{(\ell)}(u)=\psi_0(u)g^{(\ell)}(u)+\sum_{j=\ell}^{\ell+1}\binom{\ell}{j-1} \psi_1^{(\ell+1-j)}(u)g^{(j)}(u).
$$
Since $\psi_1(u)=0$, the above equality is rewritten as
$$
\lambda g^{(\ell)}(u)=\psi_0(u)g^{(\ell)}(u)+\ell\psi_1'(u)g^{(\ell)}(u),
$$
which gives, as $g^{(\ell)}(u)\ne 0$, the explicit form of the eigenvalue.
\end{proof}

To find the point spectrum, we need to find which of the numbers discussed in Proposition \ref{V-point-spec} are really eigenvalues. To do that, we look at the action of the adjoint of a differential operator on the kernel functions. With the help of Lemma \ref{bcajkf} we can compute the point spectrum of the adjoint.
\begin{prop}\label{V*-point-spec}
Let $\calE$ be a first-order differential operator (not necessarily maximal). Suppose that $\calE$ is densely defined. If the top coefficient $\psi_1$ has a zero at $u$ of order $1$, then
$$\overline{\psi_0(u)}+\ell\overline{\psi_1'(u)},\quad\ell\in\N$$
are eigenvalues of $\calE^*$.
\end{prop}
\begin{proof}
Let $m>1$. Let $\mathbb{K}_m$ be the algebraic linear span of $K_u, K_u^{[1]},...,K_u^{[m]}$. As the top coefficient $\psi_1$ has a zero at $u$ of order $1$, by \eqref{T*K_z}, the matrix representing of $\calE^*$ restricted to $\mathbb{K}_m$, is
\begin{displaymath}
X_m =
\left( \begin{array}{cccc}
\overline{\psi_0(u)} & * & \ldots & *\\
0 & \overline{\psi_0(u)+\psi_1'(u)} & \ldots & *\\
0 & 0 & \ldots & * \\
\vdots & \vdots & \ddots & \vdots\\
0 & 0 & \ldots & \overline{\psi_0(u)+m\psi_1'(u)}
\end{array} \right).
\end{displaymath}
Since any finite dimensional subspace is closed, so is $\mathbb{K}_m$. Hence, we can write $H^2=\mathbb{K}_m\oplus\mathbb{K}_m^\bot$. The block matrix of $\calE^*$ corresponding to this decomposition is
\begin{displaymath}
\left( \begin{array}{cc}
X_m & Y_m\\
0 & Z_m
\end{array} \right).
\end{displaymath}
Thus, $\sigma(\calE^*)=\sigma(X_m)\cup\sigma(Z_m)$, and hence the spectrum $\sigma(\calE^*)$ contains the points $\overline{\psi_0(u)+m\psi_1'(u)}$.
\end{proof}

A combination of Propositions \ref{V-point-spec}-\ref{V*-point-spec} gives the main result of this section. The proof is based on the argument that ``complex eigenvalues (if any) always exist in complex-conjugate pairs".
\begin{thm}\label{general}
Let $\calE$ be a first-order differential operator (not necessarily maximal) given by
$$
\calE f=\psi_0 f+\psi_1 f',\quad f\in\text{dom}(\calE),
$$
where $\psi_0,\psi_1$ are holomorphic functions. Suppose that $\calE$ is $\calC$-selfadjoint with respect to an arbitrary conjugation. If the top coefficient $\psi_1$ has a zero at $u$ of order $1$, then
$$\sigma_p(\calE)=\left\{\overline{\psi_0(u)+k\psi_1'(u)},\quad k\in\N\right\}.$$
\end{thm}

\section*{Acknowledgments}
The author would like to thank the anonymous referee for valuable comments.

\bibliographystyle{plain}
\bibliography{refs}
\end{document}